\documentclass[11pt,a4paper]{article}

\usepackage{amsmath,amsthm,amssymb, url}
\usepackage[margin=2cm]{geometry}

\newtheorem{theorem}{Theorem}

\newtheorem{corollary}{Corollary}

\newtheorem{lemma}{Lemma}
\newtheorem{problem}{Problem}
\newtheorem{proposition}{Proposition}

\newtheorem{claim}{Claim}

\newcommand{\BF}[1]{{\bf\boldmath{#1}\unboldmath}}

\newcommand{\NIn}{N_{\mathrm{in}}}

\newcommand{\IG}{\mathrm{IG}}

\newcommand{\functions}{\mathrm{F}}

\newcommand{\GF}{\mathrm{GF}}

\newcommand{\ima}{\mathrm{ima}}
\newcommand{\per}{\mathrm{per}}

\newcommand{\Ima}{\mathrm{Ima}}
\newcommand{\Per}{\mathrm{Per}}

\begin{document}
\title{Maximum Rank and Periodic Rank of Finite Dynamical Systems}
\author{Maximilien Gadouleau\footnote{School of Engineering and Computing Sciences, Durham University, Durham, UK. Email: \texttt{m.r.gadouleau@durham.ac.uk}}}
%\classno{05C38 (primary), 05C90, 37N25, 06E30}

\maketitle

\begin{abstract}
A finite dynamical system is a system of multivariate functions over a finite alphabet used to model a network of interacting entities. The main feature of a finite dynamical system is its interaction graph, which indicates which local functions depend on which variables; the interaction graph is a qualitative representation of the interactions amongst entities on the network. The rank of a finite dynamical system is the cardinality of its image; the periodic rank is the number of its periodic points. In this paper, we determine the maximum rank and the maximum periodic rank of a finite dynamical system with a given interaction graph over any non-Boolean alphabet. We also obtain a similar result for Boolean finite dynamical systems (also known as Boolean networks) whose interaction graphs are contained in a given digraph. We then prove that the average rank is relatively close (as the size of the alphabet is large) to the maximum. The results mentioned above only deal with the parallel update schedule. We finally determine the maximum rank over all block-sequential update schedules and the supremum periodic rank over all complete update schedules.
\end{abstract}

\section{Introduction}

\BF{Finite Dynamical Systems} (FDSs) have been used to represent networks of interacting entities as follows. A network of $n$ entities has a state $x = (x_1,\dots, x_n) \in [q]^n$, represented by a $q$-ary variable $x_v \in [q] = \{0,1,\dots,q-1\}$ on each entity $v$, which evolves according to a deterministic function $f = (f_1,\dots,f_n) : [q]^n \to [q]^n$, where $f_v : [q]^n \to [q]$ represents the update of the local state $x_v$. FDSs have been used to model gene networks (see \cite{KS08,TD90}), neural networks \cite{ADG04,Hop82}, network coding \cite{Rii07}, social interactions \cite{GT83,PS83} and more (see \cite{GM90}). 

The architecture of an FDS $f: [q]^n \to [q]^n$ can be represented via its \BF{interaction graph} $\IG(f)$, which indicates which update functions depend on which variables. More formally, $\IG(f)$ has $\{1,\dots,n\}$ as vertex set and there is an arc from $u$ to $v$ if $f_v(x)$ depends on $x_u$. In different contexts, the interaction graph is known--or at least well approximated--, while the actual update functions are not. One main problem of research on FDSs is then to predict their dynamics according to their interaction graphs. However, due to the wide variety of possible local functions, determining properties of an FDS given its interaction graph is in general a difficult problem.

For instance, maximising the number of fixed points of an FDS based on its interaction graph was the subject of a lot of work, e.g. in \cite{Ara08,ADG04,GRR14,Ric09,Rii07}. The logarithm of the number of fixed points is notably upper bounded by the transversal number of its interaction graph \cite{ADG04,Rii07}. This upper bound is reached for large classes of graphs (e.g. perfect graphs) but is not tight in general \cite{Rii07}. Moreover, there is a dramatic change whether we assume that the FDS has an interaction graph equal to a certain digraph or only contained in that digraph (this is the distinction between guessing number and strict guessing number in \cite{GRF15}). 

In this paper, we are interested in maximising two other very important dynamical parameters of an FDS given its interaction graph. First, the \BF{rank} of an FDS $f$ is the number of images of $f$. In particular, determining the maximum rank also determines whether there exists a bijective FDS with a given interaction graph. This is equivalent to the existence of so-called reversible dynamics, where the whole history of the system can be traced back in time. Second, because there are only a finite number of states, all the asymptotic points of $f$ are periodic. The number of periodic points of $f$ is referred to as its \BF{periodic rank}. In contrast with the situation for fixed points, we derive a bound on these two quantities which is attained for all interaction graphs and all alphabets. In particular, there exists a bijection with interaction graph $D$ if and only if all the vertices of $D$ can be covered by disjoint cycles. Moreover, we prove that our bound is attained for functions whose interaction graph is equal to a given digraph, and not only contained, for all non-Boolean alphabets. We then show that the average rank is relatively close (as $D$ is fixed and $q$ tends to infinity) to the maximum.

These results can be viewed as the discrete analogue to Poljak's matrix theorem in \cite{Pol89}, which finds the maximum rank of $M^p$, where $M$ is a real matrix with given support and $p \ge 1$. However, our results extend Poljak's result in three ways. Firstly, they hold for all functions, not only linear functions. Secondly, they explicitly determine the maximum periodic rank. Thirdly, the average rank of a real matrix cannot be properly defined, hence our result on the average rank of finite dynamical systems is completely novel.

The results mentioned above hold for the so-called parallel update schedule, where all entities update their local state at the same time, and hence $x$ becomes $f(x)$. We then study complete update schedules, where all entities update their local state at least once, and block-sequential schedules where all entities update their local state exactly once (the parallel schedule being a very particular example of block-sequential schedule). We then prove that the upper bound on the rank in parallel remains valid for any block-sequential schedule but is no longer valid for all complete schedules. We also determine the maximum periodic rank when considering all possible complete schedules. In particular, there exists a bijection with interaction graph $D$ if and only if all the vertices of $D$ belong to a cycle.

The rest of the paper is organised as follows. Section \ref{sec:max_rank} introduces some useful notation and describes our results on the maximum (periodic) rank in parallel. Section \ref{sec:average} then proves our result on the average rank. Finally, the maximum rank and periodic rank under different update schedules are investigated in Section \ref{sec:update}.

\section{Maximum (periodic) rank in parallel} \label{sec:max_rank}

\subsection{Background and notation}

Let $D = (V,E)$ be a digraph on $n$ vertices; let $V = \{1, \dots, n\}$ be its set of vertices and $E \subseteq V^2$ its set of arcs. The digraph may have loops, but no repeated arcs. The adjacency matrix $M \in \{0,1\}^{n \times n}$ has entries $m_{u,v} = 1$ if and only if $(u,v) \in E$. We denote the in-neighbourhood of a vertex $v$ in $D$ by
$$
	\NIn(v;D) = \{u \in V : (u,v) \in E\}.
$$
When there is no confusion, we shall omit the dependence on $D$. This is extended to sets of vertices: $\NIn(S) = \bigcup_{v \in S} \NIn(v)$. The out-neighbourhood is defined similarly. A source is a vertex with empty in-neighbourhood; a sink is a vertex with empty out-neighbourhood. The in-degree of $v$ is the cardinality of its in-neighbourhood and is denoted by $d_v$.

A \BF{walk} $w = (v_0,\ldots,v_p)$ is a sequence of (not necessarily distinct) vertices such that $(v_s, v_{s+1}) \in E$ for all $0 \le s \le p-1$. A path is a walk where all vertices are distinct. A cycle is a walk where only the first and last vertices are equal. We refer to $p$ as the length of the walk; a $p$-walk is a walk of length $p$. We say that two $p$-walks $w = (w_0,\ldots,w_p), w' = (w'_0,\ldots,w'_p)$ are \BF{independent} if $w_s \ne w'_s$ for all $0 \le s \le p$. We denote the maximum number of pairwise independent $p$-walks as $\alpha_p(D)$.

Edmonds gave a formula for $\alpha_1(D)$ in \cite{Edm67}, based on the K\"onig-Ore formula:
$$
	\alpha_1(D) = n - \max\{|S| - |\NIn(S)| : S \subseteq V\}.
$$ 
This was greatly generalised by Poljak, who showed that $\alpha_p(D)$ could be computed in polynomial time and who gave a formula for $\alpha_p(D)$ for all $p \ge 1$ in \cite{Pol89}. Suppose that $C_1, \dots, C_r$ and $P_1, \dots, P_s$ are vertex-disjoint cycles and paths. The cycle $C_i = (c_0, \dots, c_{l-1})$ produces $l$ independent $p$-walks of the form $W_a = (c_a, c_{a+1}, \dots, c_{a+p-1})$, where indices are computed mod $l$ and $0 \le a \le l-1$. The path $P_j = (d_0, \dots, d_{m-1})$ produces $m-p$ independent $p$-walks of the form $W_b = (d_b, d_{b+1}, \dots, d_{b+p-1})$, where $0 \le d \le m-p-1$. Poljak's theorem asserts that this is the optimal way of producing pairwise independent $p$-walks. We denote the number of vertices of a cycle $C$ and of a path $P$ as $|C|$ and $|P|$, respectively.

\begin{theorem}[\cite{Pol89}]
For every digraph $D$ and a positive integer $p$, 
$$
	\alpha_p(D) = \max \left\{ \sum_{i=1}^r |C_i| + \sum_{j=1}^s (|P_j| - p) \right\},
$$
where the maximum is taken over all families of pairwise vertex-disjoint cycles and paths $C_1, \dots, C_r$ and $P_1, \dots, P_s$.
\end{theorem}

\begin{corollary}
For all $p \ge n$,
$$
	\alpha_p(D) = \alpha_n(D) = \max \left\{ \sum_{i=1}^r |C_i| \right\},
$$
where the maximum is taken over all families of pairwise vertex-disjoint cycles.
\end{corollary}

A \textbf{finite dynamical system} is a function $f: [q]^n \to [q]^n$, where $[q] = \{0,1,\dots,q-1\}$ is a finite alphabet; we denote $f = (f_1,\ldots,f_n)$, where $f_v : [q]^n \to [q]$. The \BF{interaction graph} $\IG(f)$ is the digraph with vertex set $V = \{1, \dots, n\}$ such that $(u,v) \in E(\IG(f))$ if and only if $f_v$ depends essentially on $u$, i.e. there exist $x,y \in [q]^n$ which only differ on coordinate $u$ such that $f_v(x) \ne f_v(y)$. The set of all functions over an alphabet of size $q$ and whose interaction graph is (contained in) $D$ is denoted as
\begin{align*} 
	\functions[D, q] &:= \{f : [q]^n \to [q]^n : \IG(f) = D \},\\
	\functions(D, q) &:= \{f : [q]^n \to [q]^n : \IG(f) \subseteq D \}.
\end{align*}

We consider successive iterations of $f$; we thus denote $f^1(x) = f(x)$ and $f^{k+1}(x) = f(f^k(x))$ for all $k \ge 1$. Recall that $x$ is an \textbf{image} if there exists $y$ such that $x = f(y)$; $x$ is a \BF{periodic point} of $f$ if there exists $k \in \mathbb{N}$ such that $f^k(x) = x$. We are interested in the following quantities:
\begin{enumerate}
	\item the \BF{rank} of $f$ is the number of its images: $|\Ima(f)|$; 
	
	\item the \BF{periodic rank} of $f$ is the number of its periodic points: $|\Per(f)|$.
\end{enumerate}
It will be useful to scale these two quantities using the logarithm in base $q$:
\begin{align*}
	\ima(f) &:= \log_q |\Ima(f)|,\\
	\per(f) &:= \log_q |\Per(f)|.
\end{align*}
Moreover, the maximum (periodic) rank over all functions in $\functions[D, q]$ is denoted as
\begin{align*}
	\ima[D, q] &:= \max \{ \ima(f) : f \in \functions[D,q] \},\\ 
	\per[D, q] &:= \max \{ \per(f) : f \in \functions[D,q] \};
\end{align*}
and $\ima(D, q)$ and $\per(D, q)$ are defined similarly. We finally note that $\per(f) = \ima(f^p)$ for all $p \ge q^n - 1$. Therefore, the main strategy is to maximise the scaled rank of $f^p$ for all $p$.

\subsection{The main theorem and its proof}

\begin{theorem} \label{th:max_rank}
For all $D$, $p$, and $q \ge 3$,
$$
	\max \{ \ima(f^p): f \in \functions[D, q] \} = \max \{ \ima(f^p): f \in \functions(D, q) \} = \alpha_p(D).
$$
For all $D$, $p$,
$$
	\max \{ \ima(f^p): f \in \functions(D, 2) \} = \alpha_p(D).
$$
\end{theorem}

\begin{corollary}[Maximum rank] \label{cor:max_rank}
For all $D$ and $q \ge 3$,
\begin{align*}
	\ima[D, q] = \ima(D, q) = \alpha_1(D).
\end{align*}
For $q=2$,
$$
	\ima(D, 2) = \alpha_1(D).
$$
\end{corollary}

\begin{corollary}[Maximum periodic rank] \label{cor:max_periodic}
For all $D$ and $q \ge 3$,
$$
	\per[D, q] = \per(D, q) = \alpha_n(D).
$$
For $q=2$,
$$
	\per(D, 2) = \alpha_n(D).
$$
\end{corollary}

\begin{corollary}[Reversible dynamics in parallel] \label{cor:reversible_parallel}
For any $q \ge 3$, there exists $f \in \functions[D,q]$ which is a permutation of $[q]^n$ if and only if all the vertices of $D$ can be covered by disjoint cycles.
\end{corollary}

The case $q=2$ is indeed specific, for there exist graphs $D$ such that $\max \{\ima(f^p): f \in \functions[D, 2] \} < \alpha_p(D)$ for all $p \ge 1$. We shall investigate this in the next subsection.

The rest of this subsection is devoted to the proof of Theorem \ref{th:max_rank}. We begin with the upper bound on the scaled rank.

%\section{Upper bound on the maximum rank} \label{sec:upper_bound}

We now review the communication model based on terms from logic introduced by Riis and Gadouleau in \cite{RG11}. Let $\{x_1, \dots, x_k\}$ be a set of variables and consider a set of function symbols $\{f_1, \dots, f_l\}$ with respective arities (numbers of arguments) $d_1, \dots, d_l$. A \BF{term} is defined to be an object obtained from applying function symbols to variables recursively. We say that $u$ is a subterm of $t$ if the term $u$ appears in the definition of $t$. Furthermore, $u$ is a \BF{direct subterm} of $t$ if $t = f_j(v_1,\ldots,u,\ldots,v_{d_j})$, and we denote it by $u \prec t$.

Let $\Gamma = \{t_1, \dots,t_r\}$ be a set of terms built on variables $x_1, \dots,x_k$ and function symbols $f_1, \dots,f_l$ of respective arities $d_1,d_2,\ldots,d_l$. We denote the set of variables that occur in terms in $\Gamma$ as $\Gamma_{{\rm var}}$ and the collection of subterms of one or more terms in $\Gamma$ as $\Gamma_{\rm sub}$.
To the term set $\Gamma$ we associate the digraph $G_{\Gamma} = (V_\Gamma = \Gamma_{\rm sub}, E_\Gamma = \{(u,v) : u \prec v\})$. The set of sources in $G_\Gamma$ is $\Gamma_{{\rm var}}$ and the set of sinks is $\Gamma$. The min-cut of $\Gamma$ is the minimum size of a vertex cut of $G_\Gamma$ between $\Gamma_{{\rm var}}$ and $\Gamma$.

An \BF{interpretation} for $\Gamma$ over $[q]$ is an assignment of the function symbols $\psi = \{\bar{f}_1, \dots,\bar{f}_l\}$, where $\bar{f}_i : [q]^{d_i} \rightarrow [q]$ for all $1 \leq i \leq l$. We note that $\bar{f}_i$ may not depend essentially on all its $d_i$ variables. Once all the function symbols $f_i$ are assigned functions $\bar{f}_i$, then by composition each term $t_j \in \Gamma$ is assigned a function $\bar{t}_j : [q]^k \rightarrow [q]$. We shall abuse notations and also denote the \BF{induced mapping} of the interpretation as $\psi : [q]^k \rightarrow [q]^r$, defined as
$\psi(a) = \big(\bar{t}_1(a), \ldots,\bar{t}_r(a)\big)$.

\begin{theorem}[\cite{RG11} with our notation] \label{th:bound_dispersion}
Let $\Gamma$ be a term set with min-cut of $\rho$ and $\psi$ be an interpretation for $\Gamma$ over $[q]$, then $\ima(\psi) \leq \rho$.
\end{theorem}

\begin{lemma} \label{lem:upper_bound}
For any $p \ge 1$ and $\bar{f} \in \functions(D, q)$, $\ima( \bar{f}^p ) \le \alpha_p(D)$.
\end{lemma}

\begin{proof}
For all $v \in V$, denoting $\NIn(v; D) = \{u_1, \dots, u_k\}$ sorted in increasing order, we have $\bar{f}_v(x) = \bar{f}_v(x_{u_1}, \dots, x_{u_k})$. By definition, $\bar{f}^p$ is the induced mapping of an interpretation for $\Gamma^p = \{t^p_1, \dots, t^p_n\}$, where $\Gamma^0 = \{t^0_1 = x_1, \dots, t^0_n = x_n\}$ and for all $1 \le s \le p$, 
$$
	t^s_v = f_v(t^{s-1}_{u_1}, \dots, t^{s-1}_{u_k}).
$$

The graph $G_{\Gamma^p} = (V_{\Gamma^p}, E_{\Gamma^p})$ is then given by
\begin{align*}
	V_{\Gamma^p} &= \Gamma^0 \cup \dots \cup \Gamma^p\\
	E_{\Gamma^p} &= \{ (t^{s-1}_w, t^s_v) : 1 \le s \le p, w \in \NIn(v; D) \}.
\end{align*}
A flow in $G_{\Gamma^p}$ is a set of vertex-disjoint paths from $\Gamma^0$ to $\Gamma^p$. Such a path is of the form $t_W = (t^0_{w_0}, \dots, t^p_{w_p})$ where $w_{s-1} \in \NIn(w_s; D)$; it naturally induces a walk in $D$: $W = (w_0, \dots, w_p)$. Since the paths $t_W$ and $t_{W'}$ are vertex-disjoint, the corresponding walks $W$ and $W'$ are independent. Therefore, the max-flow of $G_{\Gamma^p}$ is at most $\alpha_p(D)$. By the max-flow min-cut theorem and Theorem \ref{th:bound_dispersion}, $\ima(\bar{f}^p) \le \alpha_p(D)$.
\end{proof}

%\section{Achieving the upper bound} \label{sec:lower_bound}

Let $W_1, \dots, W_\alpha$ be $\alpha := \alpha_p(D)$ independent walks of length $p$; we denote $W_i = (w_{i,0}, \ldots, w_{i,p})$ and $W = \{w_{i,s} : 1 \le i \le \alpha , 0 \le s \le p\}$. By construction, if $w$ precedes $w'$ on one walk and $w'$ appears on another walk and has a predecessor there, then $w$ precedes $w'$ in the other walk as well. We also let $U = V \setminus W$ be the set of vertices which do not belong to any of these walks and for all $0 \le s \le p-1$, we denote $W^s = \{w_{i,s} : 1 \le i \le \alpha\}$, $U^s = V \setminus W^s$ and $U' = V \setminus (W^1 \cup \dots \cup W^p)$.

We can now construct the finite dynamical systems which attain the upper bound on the scaled rank. The case $q=2$ and $f \in \functions(D, 2)$ is easy. We use a finite dynamical system where $w_{i,s+1}$ simply copies the value $x_{w_{i,s}}$; this will transmit the value $x_{w_{i,0}}$ along the walk $W_i$.

\begin{lemma}
The function $f \in \functions(D, 2)$ defined as
\begin{align*}
	f_{w_{i,s+1}}(x) &= x_{w_{i,s}} & 0 \le s \le p-1, 1 \le i \le \alpha,\\
	f_u(x) &= 0  & \text{if } u \in U',
\end{align*}
satisfies $\ima(f^p) = \alpha_p(D)$.
\end{lemma}

\begin{proof}
Let $X = \{x \in [q]^n: x_{U^0} = (0,\ldots,0)\}$; we then have $\log_q |X| = |W| = \alpha_p(D)$. It is easy to show, by induction on $s$, that for all $0 \le s \le p$, $|f^s_{W^s}(X)| = |X|$. Thus $\ima(f^p) = \alpha_p(D)$.
\end{proof}

For $q \ge 3$ and $f \in \functions[D, q]$, we use a finite dynamical system where $w_{i,s+1}$ wishes to copy the value $x_{w_{i,s}}$ whenever it can. Each other vertex $u \in \NIn(w_{i,s+1})$ has a red light (the value 2). If all lights are red, then $w_{i,s+1}$ cannot copy the value $x_{w_{i,s}}$ any more; instead it flips it from 0 to 1 and vice versa.

\begin{lemma} \label{lem:q=2}
For $q \ge 3$, the function $f \in \functions[D, q]$ defined as
\begin{align*}
	f_{w_{i,s+1}}(x) &= \begin{cases}
		1 - x_{w_{i,s}} &\text{if } x_{w_{i,s}} \in \{0,1\} \text{ and } x_{\NIn(w_{i,s+1}) \setminus w_{i,s}} = (2,\ldots,2),\\
		x_{w_{i,s}} &\text{otherwise}
	\end{cases}\\
	& 0 \le s \le p-1, 1 \le i \le \alpha,\\
	f_u(x) &= \begin{cases}
	1 &\text{if } x_{\NIn(u)} = (1,\ldots,1)\\
	0 &\text{otherwise}
	\end{cases}\\
	& \text{if } u \in U',
\end{align*}
satisfies $\ima(f^p) = \alpha_p(D)$.
\end{lemma}

\begin{proof}
The proof is similar, albeit more complex, than the one of Lemma \ref{lem:q=2}.

\begin{claim} \label{claim:1}
For all $0 \le s \le p-1$, if $x_{W^s} \ne y_{W^s}$ and $x_{U^s}, y_{U^s} \in \{0,1\}^{|U^s|}$, then $f_{W^{s+1}}(x) \ne f_{W^{s+1}}(y)$ and $f_{U^{s+1}}(x), f_{U^{s+1}}(y) \in \{0,1\}^{|U^{s+1}|}$. 
\end{claim}

\begin{proof}[Proof of Claim \ref{claim:1}]
We prove the first assertion. First, suppose there exists $w_{i,s} \in W^s$ where $x_{w_{i,s}} \ge 2$ and $x_{w_{i,s}} \ne y_{w_{i,s}}$, then 
$$
	f_{w_{i,s+1}}(x) = x_{w_{i,s}} \ne f_{w_{i,s+1}}(y).
$$
Second, suppose that for any $w_{i,s} \in W^s$ such that $x_{w_{i,s}} \ne y_{w_{i,s}}$, we have $\{x_{w_{i,s}}, y_{w_{i,s}}\} = \{0,1\}$. Then
\begin{align*}
	f_{w_{i,s+1}}(x) = 1 - x_{w_{i,s}} &\Leftrightarrow x_{\NIn(w_{i,s+1}) \setminus w_{i,s}} = (2,\ldots,2)\\
	&\Leftrightarrow (\NIn(w_{i,s+1}) \subseteq W^s) \land (y_{\NIn(w_{i,s+1}) \setminus w_{i,s}} \ne (2,\ldots,2))\\
	&\Leftrightarrow f_{w_{i,s+1}}(y) = 1 - y_{w_{i,s}}.
\end{align*}

For the second assertion, let $v \in U^{s+1}$, then either $v \in U'$ or $v = w_{i,t+1}$ with $0 \le t \ne s$. If $v \in U'$, then $f_v(x) \in \{0,1\}$ for any $x$. Suppose that $v = w_{i,t+1}$ such that $f_{w_{i,t+1}} \notin \{0,1\}$. Then $x_{w_{i,t}} \notin \{0,1\}$, which implies $w_{i,t} \in W^s$, say $w_{i,t} = w_{j,s}$; but then, $v = w_{j,s+1} \notin U^{s+1}$.
\end{proof}

Let $X = \{x \in [q]^n: x_{U^0} = (0,\ldots,0)\}$; we then have $\log_q |X| = |W| = \alpha_p(D)$.

\begin{claim}\label{claim:2}
For all $0 \le s \le p$, $|f^s_{W^s}(X)| = |X|$ and for any $x \in X$, $f^s_{U^s}(x) \in \{0,1\}^{|U^s|}$.
\end{claim}

\begin{proof}[Proof of Claim \ref{claim:2}]
The proof is by induction on $s$; the statement is clear for $s=0$. Suppose it holds for up to $s$. For any distinct $x,y \in X$, we have $f^s_{W^s}(x) \ne f^s_{W^s}(y)$ and $f^s_{U^s}(x), f^s_{U^s}(y) \in \{0,1\}^{|U^s|}$. By Claim \ref{claim:1}, we obtain that $f^{s+1}_{W^{s+1}}(x) \ne f^{s+1}_{W^{s+1}}(y)$ and $f^{s+1}_{U^{s+1}}(x) \in \{0,1\}^{s+1}$.
\end{proof}
\end{proof}

\subsection{Maximum rank in the Boolean case}

We first exhibit a class of digraphs for which the upper bound on the rank is not reached in the Boolean case.

\begin{proposition} \label{prop:not_reaching_the_bound}
Let $D$ be a digraph such that $\alpha_1(D) = n$ and $d_v = 2$ for all vertices $v \in V$. Then $\ima(f^p) < \alpha_p(D)$ for all $f \in \functions[D, 2]$ and all $p \ge 1$.
\end{proposition}

\begin{proof}
Suppose $f \in \functions[D, 2]$ is a permutation of $\{0,1\}^n$, then all the local functions $f_v$ must be balanced, i.e. $|f_v^{-1}(0)| = |f_v^{-1}(1)|$ for all $v \in V$. Because the in-degree of $v$ is equal to two, say $\NIn(v) = \{u_1, u_2\}$, we must have $f_v(x_{u_1}, x_{u_2}) = x_{u_1} + x_{u_2} + c_v$, where $c_v \in \GF(2)$. Therefore, $f(x) = Mx + c$, but since every vertex has even in-degree, the sum of all rows in $M$ (in $\GF(2)$) equals zero and $M$ is singular.
\end{proof}

For instance, if $D$ is the undirected cycle on $n$ vertices, or the directed cycle on $n$ vertices with a loop on each vertex, then for all $p \ge 1$, $\alpha_p(D) = n$  but $\ima(f^p) < n$ for all $f \in \functions[D, 2]$.

It is unknown whether there exist other such examples. On the other hand, we can easily exhibit a class of digraphs which do reach the bound. For instance, let $D = \mathring{K_n}$ be the clique with a loop on each vertex (alternatively, $E = V^2$). Then the following $f \in \functions[\mathring{K_n}, 2]$ is a permutation:
$$
	f_v(x) = \begin{cases}
	x_v &\text{if } x \notin \{ (0,\dots,0), (1,\dots,1) \}\\
	x_v + 1 &\text{otherwise};
	\end{cases} 
$$
indeed $f$ is the transposition of $(0,\dots,0)$ and $(1,\dots,1)$. Less obviously, the clique also admits a permutation of $\{0,1\}^n$.

\begin{proposition} \label{prop:Kn}
For any $n \ne 3$, $\ima[K_n, 2] = n$.
\end{proposition}

\begin{proof}
Firstly, let $n$ be even. Then we claim that $f(x) = Mx$ is a permutation, or equivalently that $\det(M) = 1$. For $\det(M) = d(n) \mod 2$, where $d(n)$ is the number of derangements (fixed point-free permutations) of $[n]$. Enumerating the permutations of $[n]$ according to their number $p$ of fixed points, we have
$$
	n! = d(n) + \sum_{p=1}^{n-1} \binom{n}{p} d(n-p) + 1.
$$ 
Since $n!$ and $\binom{n}{1}, \dots, \binom{n}{n-1}$ are all even, it follows that $d(n)$ is odd, thus $\det(M) = 1$.

Secondly, let $n \ge 5$ be odd. We prove the result by induction on $n$ odd. Let us settle the case where $n = 5$. We construct $f \in \functions[K_5, 2]$ as follows:
\begin{align*}
	(f_1, f_2, f_3)(x) &= 
	\begin{cases}
		(x_3, x_1, x_2) &\text{if } x_4 = x_5\\
		(x_2, x_3, x_1) &\text{otherwise},
	\end{cases}\\
	(f_4, f_5)(x) &= 
	\begin{cases}
		(x_5, x_4) &\text{if } (x_1, x_2, x_3) = (0,0,0)\\
		(x_5 + 1, x_4 + 1) &\text{otherwise}.
	\end{cases}
\end{align*}
It is easy to check that $f$ is a permutation of $[2]^5$.

The inductive case is similar. Suppose that $g \in \functions[K_n, 2]$ is a permutation, then construct $f \in \functions[K_{n+2}, 2]$ as follows:
\begin{align*}
	(f_1, \dots, f_n)(x) &= 
	\begin{cases}
		g(x_1, \dots, x_n) &\text{if } x_{n+1} = x_{n+2}\\
		g(x_1, \dots, x_n) + (1, \dots, 1) &\text{otherwise},
	\end{cases}\\
	(f_{n+1}, f_{n+2})(x) &= 
	\begin{cases}
		(x_{n+2}, x_{n+1}) &\text{if } (x_1, \dots, x_n) = (0,\dots,0)\\
		(x_{n+2} + 1, x_{n+1} + 1) &\text{otherwise}.
	\end{cases}
\end{align*}
Again, it is easy to check that $f$ is a permutation of $[2]^n$.
\end{proof}

\section{Average rank} \label{sec:average}

\begin{theorem} \label{th:average_rank}
The average scaled rank in $\functions[D, q]$ tends to $\alpha_1(D)$:
$$
	\lim_{q \to \infty} \frac{1}{|F[D, q]|}  \sum_{f \in F[D, q]} \ima(f) = \alpha_1(D).
$$
\end{theorem}

\begin{proof}
Let $a := \alpha_1(D)$ and $(u_1, v_1), \dots, (u_a, v_a)$ be a collection of pairwise independent arcs. Let $q$ be large enough and $f$ be chosen uniformly at random amongst $\functions[D, q]$. Let $h^0 = (x_{u_1}, \dots, x_{u_a})$ and for any $1 \le i \le a$, let 
$$
	h^i = (f_{v_1}, \dots, f_{v_i}, x_{u_{i+1}}, \dots, x_{u_a}) : [q]^n \to [q]^a.
$$
Let $c_i$ be defined as $c_0 = 1$ and $c_i = \frac{c_{i-1}^2}{8}$ for $1 \le i \le a$. 

Since $|\Ima(f)| \ge |\Ima(h^a)|$, all we need is to prove the following claim: with high probability, $|\Ima(h^i)| \ge c_i q^a$.

The proof is by induction on $i$. The claim clearly holds for $i=0$; suppose it holds for $i$. Let $g = (f_{v_1}, \dots, f_{v_i}, x_{u_{i+2}}, \dots, x_{u_a}) : [q]^n \to [q]^{a-1}$. Then we have
\begin{align*}
	Z &:= \left\{ z \in \Ima(g) : |(z, x_{i+1}) \in \Ima(h^i)| \ge \frac{1}{2} c_i q \right\},\\
	|Z| &\ge \frac{1}{2} c_i q^{a - 1}.
\end{align*}
For otherwise, we would have
$$
	|\Ima(h^i)| < q |Z| + \left( \frac{1}{2} c_i q \right) q^{a - 1} \le c_i q^a.
$$ 
Now let $N$ be the in-neighbourhood of $v_{i+1}$; note that $u_{i+1} \in N$. Therefore, for each $z \in Z$, there exist at least $\frac{1}{2} c_i q$ values of $x_N$ such that $z = g(x_N)$; denote this set of values as $X$. On $X$, $f_{v_{i+1}}(x_N)$ is chosen uniformly at random. With probability exponentially small, we have $|f_{v_{i+1}}(X)| \le \frac{1}{2} |X|$, since
\begin{align*}
	P \left( |f_{v_{i+1}}(X)| \le \frac{1}{2} |X| \right) 
	& \le \binom{q}{|X|/2} \left( \frac{|X|}{2 q} \right)^{|X|}\\
	& \le \left( \frac{2 e q}{|X|} \right)^{|X|/2} \left( \frac{|X|}{2 q} \right)^{|X|}\\
	& = \left( \frac{e |X|}{2 q} \right)^{|X|/2}\\
	& \le \left( \frac{e c_i}{4} \right)^{c_i q/4}.
\end{align*}
Therefore, with high probability, $|f_{v_{i+1}}(X)| > \frac{1}{2} |X|$ for all $z \in Z$, and hence
$$
	|\Ima(h^i)| \ge |Z| \frac{1}{4} c_i q \ge c_{i+1} q^a.
$$
\end{proof}

We make two remarks about Theorem \ref{th:average_rank}. 

Firstly, the theorem only gives an approximation of the average rank. Obtaining more detailed information seems difficult, because the average rank can vary widely with the digraph $D$. For instance, let us compare the complete graph with $n$ loops $\mathring{K_n}$ to the empty graph with $n$ loops $\mathring{L_n}$; both graphs have $\alpha_1(D) = n$. It is well known that the average rank of a function $[r] \to [r]$ tends to $\epsilon r$, where $\epsilon = 1 - e^{-1}$. Then the average rank in $\functions[\mathring{K_n}, q]$ tends to $\epsilon q^n$, while in $\functions[\mathring{L_n}, q]$ it tends to $\epsilon^n q^n$.

Secondly, there is no analogue of the theorem for the periodic rank. Again, let us use $\mathring{K_n}$. The average periodic rank of a 
function $[r] \to [r]$ tends to $\delta \sqrt{r}$, where $\delta = \sqrt{\pi/2}$. Then, the average rank in $\functions[\mathring{K_n}, q]$ tends to $\delta q^{n/2} = o(q^{\alpha_n(\mathring{K_n})})$.

\section{Maximum (periodic) rank under different update schedules} \label{sec:update}

An update schedule, or simply schedule, corresponds to the way the different entities of the underlying network represented by $f$ update their local state. More formally, a \textbf{schedule} for $f \in \functions[D, q]$ is any $\sigma = (\sigma_1, \dots, \sigma_t)$ where $\sigma_i \subseteq V$. We denote the application of $f$ using the schedule $\sigma$ as $f^\sigma$: for any $S \subseteq V$, we let $f^{(S)}$ where
$$
	f_v^{(S)}  = 
	\begin{cases}
		f_v(x) &\text{if } v \in S,\\
		x_v &\text{otherwise},
	\end{cases}
$$
and
$$
	f^{\sigma} = f^{\sigma_t} \circ \dots \circ f^{\sigma_1}.
$$
We now review three important classes of schedules. 
\begin{enumerate}
	\item $\sigma$ is \textbf{complete} if every entity updates its local state at least once, i.e. if $\bigcup_{i=1}^t \sigma_i = V$. 

	\item $\sigma$ is \textbf{block-sequential} if every entity updates its local state exactly once, i.e. if $\bigcup_{i=1}^t \sigma_i = V$ and $\sigma_i \cap \sigma_j = \emptyset$ for all $i \ne j$.

	\item $\sigma$ is \textbf{parallel} if all entities update their state once and at the same time, i.e. if $\sigma = (V)$. Clearly, $f^{(V)} = f$.
\end{enumerate}

We first prove that the $\alpha_1(D)$ upper bound on the scaled rank remains valid for block-sequential schedules.

\begin{theorem} \label{th:max_rank_block}
If $\sigma$ is a block-sequential schedule and $f \in \functions[D, q]$, then $\ima(f^\sigma) \le \alpha_1(D)$.
\end{theorem}

\begin{proof}
We use a proof technique similar to that of Theorem \ref{lem:upper_bound}. Let $\sigma = (\sigma_1, \dots, \sigma_t)$ be a block-sequential schedule. Construct the term set $\Gamma$ built on $x_1, \dots, x_n$ and the $n+1$ function symbols $f_1, \dots, f_n, g$, where $f_i$ is $d_i$-ary and $g$ is unary, uniquely defined as such.
\begin{enumerate}
	\item The subterm graph $G_\Gamma = (V_\Gamma, E_\Gamma)$ is as follows: $V_\Gamma = V^0 \cup \dots \cup V^t$ consists of $t+1$ copies of $V$, and $(u^{i-1}, v^i) \in E_\Gamma$ if either $(u,v) \in E$ and $v \in \sigma_i$ or $u = v$ and $v \notin \sigma_i$. 
	
	\item On $v^i$, $\Gamma$ uses the function symbol $f_v$ if $v^i \in \sigma_i$ and the function symbol $g$ if $v^i \notin \sigma_i$. 
\end{enumerate}
Then it is clear that for any $\bar{f} \in \functions(D, q)$, $\bar{f}^\sigma$ can be viewed as an interpretation of $\Gamma$, where $g$ is interpreted as the identity. Therefore, $\ima(\bar{f}^\sigma)$ is no more than the min-cut of $\Gamma$.

All that is left is to show that $G_\Gamma$ has at most $\alpha_1(D)$ disjoint paths from $V^0$ to $V^t$. Let $P_1, \dots, P_m$ be a family of disjoint paths starting at vertices $1, \dots, m$ and let $v_1, \dots, v_m$ be the ``first updated vertices'' on the respective paths. Formally, let $P_i = (w_0^0, \dots, w^t_t)$, where $w_0 = i$, let $a = \min \{ b : w_b^b \in \sigma_b \}$ (such $a$ always exists) and $v_i = w_a$.

Then for $i \ne j$, we have: $(i, v_i)$ and $(j, v_j)$ are arcs in $D$, $i \ne j$, and $v_i \ne v_j$ (clear if $v_i$ and $v_j$ are in different parts of $\sigma$, otherwise if $v_i, v_j \in \sigma_a$ then because the paths are disjoint we have $v_i^a \ne v_j^a$). In other words, $(1, v_1), \dots, (m, v_m)$ are independent arcs in $D$, thus $m \le \alpha_1(D)$.
\end{proof}

In particular, we can refine Corollary \ref{cor:reversible_parallel} on the presence of reversible dynamics.

\begin{corollary} \label{cor:reversible_block}
For any $q \ge 3$, the following are equivalent:
\begin{enumerate}
	\item $\functions[D,q]$ contains a permutation of $[q]^n$,
	
	\item there exist $f \in \functions[D,q]$ and a block-sequential schedule $\sigma$ such that $f^\sigma$ is a permutation of $[q]^n$,
	
	\item all the vertices of $D$ can be covered by disjoint cycles.
\end{enumerate}
\end{corollary}

\begin{problem}
Is there an analogue of Theorem \ref{th:max_rank_block} for the periodic rank?
\end{problem}

However, the maximum rank when considering any complete schedule is not bounded by $\alpha_1(D)$. In fact, the periodic rank can be much larger, as seen below. Recall that a strong component of a digraph is \textbf{trivial} if it has no cycle, or equivalently if it is a single loopless vertex. Clearly, a vertex $v$ belongs to a cycle of $D$ if and only if $\{v\}$ is not a trivial strong component of $D$. We denote the number of trivial strong components of $D$ as $T(D)$.

\begin{theorem} \label{th:max_periodic_rank_complete}
For all $D$, 
$$
	\sup \{ \per(f^\sigma) : q \ge 2, f \in \functions[D, q], \sigma \text{complete} \}  = n - T(D).
$$
\end{theorem}

\begin{proof}
Let $f \in \functions[D, q]$ and $\sigma$ be a complete schedule. Then $(u,v)$ is an arc of $\IG(f^\sigma)$ only if there is a path from $u$ to $v$ in $D$. Consequently, if $\{v\}$ is a trivial strong component of $D$, then $\{v\}$ is a trivial strong component of $\IG(f^\sigma)$. By Corollary \ref{cor:max_periodic}, we have $\per(f^\sigma) \le n - T(D)$.

Conversely, let $C_1, \dots, C_k$ be a collection of cycles which cover all vertices belonging to a cycle, $W$ denote the set of remaining vertices and let $\sigma = (W, C_1, \dots, C_k)$. Let $q-1 = 2^m$ be large enough ($m \ge 2^{n^2 + 1}$) and let $\alpha$ be a primitive element of $\GF(q-1)$. Denote the arcs in $D$ as $e_1, \dots, e_l$. Let $A \in \GF(q - 1)^{n \times n}$ such that $a_{u,v} = \alpha^{2^i}$ if $(u,v) = e_i$ and $a_{u,v} = 0$ if $(u,v) \notin E$ and let $g(x) = Ax$. Now $f \in \functions[D, q]$ is given as follows: view $[q] = \GF(q-1) \cup \{q-1\}$ and
\begin{align*}
	f_w(x) &= \begin{cases}
	0 & \text{if } x_u \in \GF(q-1) \text{ for all }  u \in \NIn(w),\\
	q-1 & \text{otherwise},
	\end{cases} \quad \forall w \in W\\
	f_v(x) &= \begin{cases}
	g_v(x) & \text{if } x_u \in \GF(q-1) \text{ for all } u \in \NIn(v),\\
	q-1 & \text{otherwise},
	\end{cases} \quad \forall v \notin W.
\end{align*}

Then $f$ acts like $g$ on the set of states $X = \{x \in \GF(q-1)^n : x_W = (0,\dots,0) \}$; in particular, we have $f(X) \subseteq X$. We can then remove $W$ and consider $h \in \functions[D \setminus W, q-1]$ such that $h_v(x_{V \setminus W}) = g_v(x_{V \setminus W}, 0_W)$ for all $v \notin W$ instead. All we need to prove is that $h^{(C_1, \dots, C_k)}$ is a permutation of $\GF(q-1)^{n-T(D)}$.

Denote the square submatrix of $A$ induced by the vertices of $C_j$ as $A_j$. Then we remark that $\det(A_j) \ne 0$ for any $1 \le j \le k$. Indeed, let $K_1, \dots, K_l$ denote all the hamiltonian cycles in the subgraph induced by the vertices of $C_j$ (and without loss, $K_1 = C_j$). For any $1 \le a \le l$, let $S(a) = \sum_{e_i \in K_l} 2^i$. We note that $S(1), \dots, S(l)$ are all distinct, hence $\alpha^{S(1)}, \dots, \alpha^{S(l)}$ are all linearly independent (when viewed as vectors over $\GF(2)$) and
$$
	\det(A_j) = \sum_{a=1}^l \alpha^{S(a)} \ne 0.
$$
Now $h^{(C_j)}(x) = A'_j x$, where
$$
	A'_j = \left(
	\begin{array}{c|c}
	A_j & B_j\\
	\hline
	0 & I
	\end{array}
	\right),
$$
where $(A_j | B_j)$ are the rows of $A$ corresponding to $C_j$ and $I$ is the identity matrix of order $n -T(D) - |C_j|$. Since $A_j$ is nonsingular, so is $A'_j$. Hence $h^{(C_j)}$ is a permutation of $\GF(q-1)^{n-T(D)}$, and by composition, so is $h^{(C_1, \dots, C_k)}$. 
\end{proof}

If $W$ is empty, then we can simplify the proof of Theorem \ref{th:max_periodic_rank_complete} and work with $\GF(q)^n$ instead of $\GF(q-1)^{n - T(D)}$ (this time $q = 2^p$), hence we obtain a permutation. This yields the following corollary on the presence of reversible dynamics.

\begin{corollary} \label{cor:permutation_complete}
There exist $q$, $\sigma$ and $f \in \functions[D, q]$ such that $f^\sigma$ is a permutation of $[q]^n$ if and only if all the vertices of $D$ belong to a cycle.
\end{corollary}

The theorem brings the following natural question.

\begin{problem} \label{pb:rank}
Is there an analogue of Theorem \ref{th:max_periodic_rank_complete} for the rank?
\end{problem}

\section{Acknowledgment}
This work was supported by the Engineering and Physical Sciences Research Council [grant number EP/K033956/1].

%\bibliographystyle{amsplain}
%\bibliography{g}
%

\providecommand{\bysame}{\leavevmode\hbox to3em{\hrulefill}\thinspace}
\providecommand{\MR}{\relax\ifhmode\unskip\space\fi MR }
\providecommand{\MRhref}[2]{%
  \href{http://www.ams.org/mathscinet-getitem?mr=#1}{#2}
}
\providecommand{\href}[2]{#2}

\end{document}